\documentclass[a4j]{amsart}
\usepackage{amsmath, amssymb, amsfonts, amsthm, mathrsfs}
\usepackage{setspace} 
\usepackage{mathdots, color, comment, graphics}

\newtheorem{thm}{Theorem}[section]
\newtheorem{cor}[thm]{Corollary}
\newtheorem{lem}[thm]{Lemma}
\newtheorem{prop}[thm]{Proposition}

\theoremstyle{definition}

\newtheorem{rem}{Remark}[section]

\theoremstyle{remark}

\makeatletter
\def\th@plain{\upshape}
\makeatother

\def\rm{\mathrm}
\def\C{\mathbb{C}}
\def\R{\mathbb{R}}
\def\Z{\mathbb{Z}}
\def\A{\mathbb{A}}
\def\GL{\mathrm{GL}}
\def\SL{\mathrm{SL}}
\def\O{\mathrm{O}}
\def\SO{\mathrm{SO}}
\def\U{\mathrm{U}}
\def\M{\mathrm{M}}

\def\mF{\mathcal{F}}

\def\Sc{\mathcal{S}}
\def\rm{\mathrm}

\def\bs{\backslash}
\def\t{\,{}^t\!}
\def\d{\mathrm{diag}}
\def\det{\mathrm{det}} 
\def\Tr{\mathrm{Tr}}
\def\N{\mathrm{N}}

\def\Sf{\mathfrak{S}}

\def\ve{\varepsilon}
\def\ga{\gamma}
\def\de{\delta}
\def\la{\lambda}
\def\ka{\kappa}
\def\th{\theta}

\def\1{\mathbf{1}}
\def\0{\mathbf{0}}
\def\idots{\reflectbox{$\ddots$}}
\def\ul{\underline}
\def\wt{\widetilde}

\DeclareMathOperator{\Res}{Res}
\DeclareMathOperator{\As}{As}
\DeclareMathOperator{\Ad}{Ad}
\DeclareMathOperator{\St}{St}
\DeclareMathOperator{\Sp}{Sp}
\DeclareMathOperator{\Irr}{Irr}
\DeclareMathOperator{\ind}{ind}
\DeclareMathOperator{\Hom}{Hom}

\begin{document}
\title[QUATERNION DISTINGUISHED]{Quaternion distinguished generic representations of $\GL_{2n}$}
\author[M. SUZUKI]{Miyu Suzuki}
\address{Graduate~school~of~Mathematics, Kyoto~University,\\
 Kitashirakawa, Kyoto, 606-8502, Japan}
\email{msuzuki@math.kyoto-u.ac.jp} 

\begin{abstract}
Let $E/F$ be a quadratic extension of non-Archimedean local fields of characteristic 0.
Let $D$ be the unique quaternion division algebra over $F$ and fix an embedding of $E$ to $D$.
Then, $\GL_m(D)$ can be regarded as a subgroup of $\GL_{2m}(E)$.
Using the method of Matringe, we classify irreducible generic $\GL_m(D)$-distinguished representations of $\GL_{2m}(E)$ in terms of Zelevinsky classification.
Rewriting the classification in terms of corresponding representations of the Weil-Deligne group of $E$, we prove a sufficient condition for a generic representation in the image of the unstable base change lift from the unitary group $\U_{2m}$ to be $\GL_m(D)$-distinguished.
\end{abstract}
\maketitle

\section*{Introduction}
Let $F$ be a non-Archimedean local field of characteristic 0.
Let $G$, $H$ be algebraic groups over $F$ and suppose $H$ is a closed subgroup of $G$.
A smooth representation $\pi$ of $G(F)$ is said to be \textit{$H(F)$-distinguished} if it has a nonzero $H(F)$-invariant linear form.
This paper concentrate on the case of $(G, H)=(\Res_{E/F}\GL_{2m}, \GL_m(D))$, where $E$ is a quadratic extension of $F$, $\Res_{E/F}$ denotes the restriction of scalars and $D$ is the unique quaternion division algebra over $F$.
Given an embedding of $E$ to $D$, $H$ can be regarded as a subgroup of $G$.
We classify irreducible smooth generic $H(F)$-distinguished representations of $G(F)$ and relate it to a functorial lift from the quasi-split unitary group $\U_{2m}$.

In order to explain the background, we consider the global setting.
Let $E/F$ be a quadratic extension of number fields and $\A_E$, $\A_F$ the rings of adeles of $E$ and $F$, respectively.
Let $\pi$ be an irreducible cuspidal automorphic representation of $\GL_{2m}(\A_E)$.
Then, $\pi$ is said to be $\GL_{2m}$-distinguished if there is a cusp form $f$ in the space of $\pi$ which has a nonzero period integral over $\GL_{2m}$:
\[
\int_{Z(\A_F)\GL_{2m}(F)\bs\GL_{2m}(\A_F)}f(h)\,dh\neq0.
\]
Here, $Z$ denotes the center of $\GL_{2m}$.
Flicker and Rallis (see \cite{flicker1991distinguished}) conjectured that $\pi$ is $\GL_{2m}$-distinguished if and only if it is an unstable base change lift of a generic cuspidal automorphic representation of the quasi-split unitary group $\U_{2m}(\A_F)$.

Let $D$ be a quaternion algebra over $F$ which $E$ embedds.
Then, distinguished cuspidal representations of $\GL_{2m}(\A_E)$ with respect to $\GL_m(D)$ is defined similarly as above.
In \cite{suzuki2019quaternion}, the author defined a non-degenerate character $\th'_{\tau}$ associated to $D$ for the quasi-split unitary group and conjectured that $\pi$ is $\GL_m(D)$-distinguished if and only if it is an unstable base change lift of a $\th'_1$-generic and $\th'_{\tau}$-generic cuspidal representation of $\U_{2m}(\A_F)$.
This conjecture slightly generalizes that of Flicker and Rallis.

Now, we go back to the local setting.
Recall that $G=\Res_{E/F}\GL_{2m}$.
Let $WD_F$ be the Weil-Deligne group of $F$ and ${}^LG$, ${}^L\U_{2m}$ be the $L$-groups of $G$ and $\U_{2m}$.
For an irreducible smooth representation $\pi$ of $G(F)=\GL_{2m}(E)$, there corresponds an $L$-parameter $\tilde{\phi}\,\colon WD_F\rightarrow{}^LG$ by local Langlands correspondence for $G$ which was established by Harris, Taylor \cite{harris2001geometry} and Henniart \cite{henniart2000preuve}.
See section 1 for more details.

According to the local Langlands correspondence in a form proposed by Vogan, for an $L$-parameter $\phi'\,\colon WD_F\rightarrow{}^L\U_{2m}$ of $\U_{2m}$, we should obtain a finite set $\Pi_{\phi'}$ of irreducible smooth representations of pure inner forms of $\U_{2m}$ called the Vogan \textit{$L$-packet} of $\phi'$.
In section 2, we will review the definition of the unstable base change map $bc_\chi$ from ${}^L\U_{2m}$ to ${}^LG$.
For an $L$-parameter $\phi'$ of $\U_{2m}$, the composition $bc_\chi\circ\phi'$ is an $L$-parameter of $G$.
Hence, if we assume the local Langlands correspondence for $\U_{2m}$, we get a map from the set of irreducible smooth representations of (pure inner forms of) $\U_{2m}(F)$ to that of $G(F)$, sending the representations in $\Pi_{\phi'}$ to the representation corresponding to $bc_\chi\circ\phi'$.
We call this conjectural map the \textit{unstable base change lift}.

Many researchers studied the following local analogue of the conjecture of Flicker and Rallis:
\begin{quote}\textbf{Statement.}
Let $\pi$ be an irreducible representation of $G(F)=\GL_{2m}(E)$.
Let $\tilde{\phi}$ be the corresponding $L$-parameter.
Then, $\pi$ is $\GL_{2m}(F)$-distinguished if and only if there exists an $L$-parameter $\phi'\,\colon WD_F\rightarrow{}^L\U_{2m}$ of $\U_{2m}(F)$ so that $\tilde{\phi}=bc_\chi\circ\phi'$.
\end{quote}
The reader may see, for example \cite{anandavardhanan2005distinguished}.
Though this naive analogue is not valid in general, it is still expected to be true for a wide class of representations.
For example, this statement is true for generic representations (see \cite[Theorem 4.9]{gurevich2018two} and \cite{matringe2010distinguished}).

In this paper, we consider the local analogue of \cite[Conjecture 2]{suzuki2019quaternion} and characterize $\GL_m(D)$-distinguished generic representations of $\GL_{2m}(E)$ in terms of an unstable base change lift of a generic representation of $\U_{2m}(F)$.
The main result is as follows:

\begin{thm}
Let $\pi$ be an irreducible generic representation of $\GL_{2m}(E)$.
Let $\tilde{\phi}$ be the corresponding $L$-parameter.
If there exists an $L$-parameter $\phi'$ of $\U_{2m}(F)$ which is generic with respect to any non-degenerate characters and satisfies $\tilde{\phi}=bc_\chi\circ\phi'$, then $\pi$ is $\GL_m(D)$-distinguished.
\end{thm}

Unlike the global conjecture, the converse implication does not hold for local representations.
In order to get more precise description of $H(F)$-distinguished representations, we need to generalize the conjecture of Prasad.
For a connected reductive quasi-split group $G$ over $F$, Prasad \cite{prasad2015relative} introduced a quasi-split group $H^{\rm{op}}$ over $F$ which is isomorphic to $H$ over $E$.
He formulated a conjectural formula \cite[Conjecture 2]{prasad2015relative} for the dimension of the space of $H(F)$-invariant linear forms on irreducible smooth representations of $H(E)$ in terms of the degree of the base change map from $H^{\rm{op}}(F)$ to $H(E)$.
In particular, he conjectured that an $H(F)$-distinguished representation of $H(E)$ arises from base change from $H^{\rm{op}}$.
For this conjecture, several low dimensional examples are known; see \cite{lu2018prasad} and \cite{lu2019prasad}.
A recent remarkable work of Beuzart-Plessis \cite{beuzart2017distinguished} established an integral formula computing the dimension of the $H(F)$-invariant linear forms for square-integrable representations.
As a corollary, he proved that this dimension is preserved under the Jacquet-Langlands correspondence.
This result verifies a consequence of the conjecture of Prasad and plays an important role in the proof of our main theorem.

In our setting, $H=\GL_m(D)$ is a non quasi-split inner form of $H'=\GL_{2m}$ and they are isomorphic over $E$.
It is natural to ask whether the dimension of the space of $H(F)$-invariant linear forms on representations of $H(E)$ can be described in terms of the base change from $H'^{\rm{op}}$. 
We hope that the result of this paper will be reinterpreted in relative Langlands frameworks in the sense of Prasad, generalized to non quasi-split groups.

We briefly explain the contents of this paper.
In section 1, we fix necessary notations and recall some facts about representations of Weil-Deligne groups.
In section 2, we review the conjectural local Langlands correspondence and known results.
In section 3, we quote some results of Matringe about double coset decomposition of $G$.
In section 4, we summarize the Zelevinsky classification of representations of $G(F)$.
In section 5, using the method of Matringe \cite{matringe2010distinguished}, we classify $H$-distinguished generic representations of $G(F)$ in terms of Zelevinsky classification.
Rewriting that classification in terms of representations of $WD_E$, we obtain the main theorem in the last section.

\section{Preliminaries}
Let $E/F$ be a quadratic extension of non-Archimedean local fields of characteristic 0.
Fix a nontrivial additive character $\psi$ of $F$.
Define the character $\psi_E$ of $E$ by
\[
\psi_E(x)=\psi\left(\frac12\Tr_{E/F}(x)\right), \hspace{10pt} x\in E.
\]
Denote the action of the nontrivial element of the Galois group $\rm{Gal}(E/F)$ by $x\mapsto x^c$.
We consider various objects on which $\rm{Gal}(E/F)$ acts.
By abuse of notation, we write all of these actions by the same symbol ${\cdot}^c$. 
Denote by $\N_{E/F}$ (resp. $\Tr_{E/F}$) the norm map (resp. trace map) from $E$ to $F$.
Let $\omega=\omega_{E/F}$ be the unique nontrivial character of $F^{\times}/\N_{E/F}(E^{\times})\cong\rm{Gal}(E/F)$.
Sometimes we regard it as a quadratic character of $F^{\times}$.

Let $W_F$ be the Weil group of $F$.
Recall that this is a subgroup of the absolute Galois group of $F$ whose Abelianization is isomorphic to $F^{\times}$.
Hence, $\omega$ can be regarded as a character of $W_F$.
Define the Weil-Deligne group $WD_F$ of $F$ by $WD_F=W_F\times\SL_2(\C)$.
Similarly, we denote by $W_E$ and $WD_E$ the Weil group and Weil-Deligne group of $E$.
Note that $W_E$ is a subgroup of $W_F$ of index $2$ and the quotient $W_F/W_E$ is naturally identified with $\rm{Gal}(E/F)$.
We fix an element $s$ of $W_F$ outside $W_E$ once for all.
The image of $s$ in $W_F/W_E$ is the nontrivial element of $\rm{Gal}(E/F)$.

\subsection{Representations of $WD_E$}
Let $M$ be a finite dimensional complex vector space.
A homomorphism $\ul{\phi}\,\colon WD_E\rightarrow\GL(M)$ is called a representation of $WD_E$ if
\begin{itemize}
\item the image of a geometric Frobenius element of $W_E$ is semisimple;
\item the restriction of $\ul{\phi}$ to $W_E$ is smooth;
\item the restriction of $\ul{\phi}$ to $\SL_2(\C)$ is algebraic.
\end{itemize}
For an $L$-parameter $\phi\,\colon WD_E\rightarrow\GL_n(\C)\times W_E$, we write by $\ul{\phi}\,\colon WD_E\rightarrow\GL(M)$ the associated representation of $WD_E$.
Here, $M$ is a complex vector space of dimension $n$.
We recall some facts about conjugate self-dual representations of $WD_E$ from \cite[Section 4]{gan2011symplectic}.
We say that the representation $M$ of $WD_E$ is \textit{conjugate self-dual} if there exists a non-degenerate bilinear form $B\,\colon M\times M\rightarrow\C$ which satisfies
\[\left\{\renewcommand{\arraystretch}{1.3}\begin{array}{l}
B(\tau m, s\tau s^{-1}n)=B(m, n) \\
B(n, m)=b\cdot B(m, s^2 n)
\end{array}\right. \hspace{10pt} \tau\in WD_E,\ m ,n\in M
\]
with some $b\in\{\pm1\}$.
We call $b$ the \textit{sign} of $B$.
If $b=+1$, $M$ is called \textit{conjugate-orthogonal} and if $b=-1$, $M$ is called \textit{conjugate-symplectic}.
For a conjugate self-dual representation $(\ul{\phi}, M)$ of $WD_E$, we have a decomposition
\begin{equation}\label{eq:WDdecomp}
M=\bigoplus_{i\in I_+}V_i\otimes M_i+\bigoplus_{i\in I_-}V_i\otimes M_i+\bigoplus_{i\in I_0}V_i\otimes(M_i+(M_i^c)^{\vee}).
\end{equation}
Here, each $M_i$ is an irreducible representation of $WD_E$ and $V_i$ is a space of multiplicity satisfying
\begin{itemize}
\item for $i\in I_+$, $M_i$ is conjugate self-dual of sign $b$;
\item for $i\in I_-$, $M_i$ is conjugate self-dual of sign$-b$;
\item for $i\in I_0$, $M_i$ is not conjugate self-dual.
\end{itemize}
The restriction of the form $B$ to each summand of the above decomposition induces a non-degenerate bilinear form on each $V_i$.
These pairings have sign $+1$ for $i\in I_+$ and $-1$ for $i\in I_-$.
Let $C=C(M, B)$ be the subgroup of $\rm{Aut}(M, B)\subset\GL(M)$ which centralizes the image of $WD_E$.
Due to the above decomposition, we have
\[
C\cong\prod_{i\in I_+}\O(V_i)\times\prod_{i\in I_-}\Sp(V_i)\times\prod_{i\in I_0}\GL(V_i).
\]
Thus, the component group of $C$ is
\[
A=A_M\cong(\Z/2\Z)^k,
\]
where $k$ is the order of $I_+$.

For a semisimple element $a$ in $C$, set 
\[
M^a=\{m\in M\, ;\, am=-m\}.
\]
Note that the parity of the dimension of $M^a$ only depends on the image of $a$ in $A$.
We define the quadratic character $\eta$ of $A$ by
\[
\eta(a)=(-1)^{\dim\,M^a}, \hspace{10pt} a\in A.
\]
In particular, $\eta$ is the trivial character if and only if $\dim\,M_i$ is even for all $i\in I_+$.

\subsection{Groups and representations}
Throughout this paper, for an algebraic group $G$ over $F$, we write the group of $F$-rational points by $G$.
Denote by $Z_G$ or simply by $Z$ the center of $G$. 

A representation of $G$ always means an admissible representation of $G$.
For a representation $(\pi, V)$ of $G$, denote its contragredient \textit{i.e.} the smooth dual by $(\pi^{\vee}, V^{\vee})$.
For a closed subgroup $H$ of $G$ and its left Haar measure $dh$, there exists a continuous character $\de_H\,\colon H\rightarrow\R_{>0}$ which satisfies
\[
d(hh')=\de_H(h')^{-1}dh
\]
for all $h'\in H$.
We call this the modulus character of $H$.
For a representation $(\sigma, W)$ of $H$, denote by $I_H^GW$ the space of locally constant functions $f\,\colon G\rightarrow W$ satisfying
\[
f(hg)=\de_H^{1/2}(h)\sigma(h)f(g)
\]
for all $h\in H$ and $g\in G$.
The representation of $G$ on this space given by right translation is called the normalized induced representation from $(\sigma, W)$ and denoted by $(I_H^G\sigma, I_H^GW)$.

Similarly, let $\ind_H^GW$ be the space of locally constant functions $f\,\colon G\rightarrow W$ satisfying
\[
f(hg)=\sigma(h)f(g)
\]
for all $h\in H$ and $g\in G$ and its support is compact modulo $H$.
The representation of $G$ on this space given by right translation is called (unnormalized) compactly induced representation from $(\sigma, W)$ and denoted by $(\ind_H^G\sigma, \ind_H^GW)$.

Write the trivial representation of $H$ by $\1_H$ or simply by $\1$.
For a character $\chi$ of $H$, we say that a representation $\pi$ of $G$ is $(H, \chi)$-\textit{distinguished} if the space $\Hom_H(\pi, \chi)$ has a nontrivial element.
If $\chi$ is the trivial character, $(H, \chi)$-distinguished representations are simply called $H$-\textit{distinguished}.

Suppose that $G$ is a quasi-split reductive group.
Take a Borel subgroup $B$ of $G$ defined over $F$.
Let $T$ be a maximal $F$-subtorus of $G$ contained in $B$ and $N$ the unipotent radical of $B$ so that we have $B=T\ltimes N$.
For a character $\th$ of $N$ and an element $t$ of $T$, we define a character $t\cdot\th$ of $N$ by 
\[
t\cdot\th(n)=\th(t^{-1}nt), \hspace{10pt} n\in N.
\]
This defines an action of $T$ on the set of characters of $N$.
A character $\th$ of $N$ is called non-degenerate or generic if its stabilizer in $T$ is $Z_G$.
The set of non-degenerate characters of $N$ is stable under the action of $T$.
Suppose that $\th$ is a non-degenerate character of $N$.
A representation $\pi$ of $G$ is called $\th$-generic if it is $(N, \th)$-distinguished.
This only depends on the $T$-orbit of $\th$.

\section{Local Langlands correspondence}

\subsection{General linear groups and Asai representations}
Set $G=G_n=\GL_n(E)$.
This can be viewed as a group of $F$-rational points of $\Res_{E/F}\GL_n$.
Here, $\Res_{E/F}$ stands for the restriction of scalars.
The finite symmetric group $\Sf_n$ is identified with the subgroup of $G$ consisting of permutation matrices.

For a partition $\la=(n_1, \ldots, n_r)$ of $n$, let $P_{\la}$ be the parabolic subgroup of $G$ consisting of matrices of the form
\[\renewcommand{\arraystretch}{1.3}
\left(\begin{array}{ccc}
g_1&\ast   &\ast \\
     &\ddots&\ast \\
     &         &g_r
\end{array}\right)\in G, \hspace{10pt} g_i\in G_{n_i}.
\]
A subgroup of this form is called a standard parabolic subgroup.
Let $M_{\la}$ be the Levi subgroup of $P_{\la}$ consisting of matrices of the form $\d(g_1, \ldots, g_r)$ with $g_i\in G_{n_i}$.
A subgroup of this form is called a standard Levi subgroup.
We write the unipotent radical of $P_{\la}$ by $N_{\la}$.
A parabolic subgroup of $M_{\la}$ is called standard if it is a product of standard parabolic subgroups of each component $G_{n_i}$.

The standard parabolic subgroup corresponding to the partition $(1, \ldots, 1)$ is a Borel subgroup of $G$ and denoted by $B$.
The standard Levi subgroup of $B$ is a maximal torus of $G$ and denoted by $T$.
Denote by $N$ the unipotent radical of $B$.

Let $\la=(n_1, \ldots, n_r)$ be a partition of $n$ and $(\sigma_i, V_i)$ a representation of $G_{n_i}$ for each $i$.
Set $(\sigma, V)=(\sigma_1\boxtimes\cdots\boxtimes\sigma_r, V_1\otimes\cdots\otimes V_r)$.
This is a representation of $M_{\la}$.
We extend $\sigma$ to a representation of $P_{\la}$ so that $N_{\la}$ acts trivially on $V$.
We write the normalized induced representation $I_{P_{\la}}^G\sigma$ by $\sigma_1\times\cdots\times\sigma_r$.

Let $P'=M'N'$ be a standard parabolic subgroup of $M_{\la}$.
We denote by $V(P')$ the subspace of $V$ spanned by elements of the form $\sigma(n')v-v$ with $v\in V$ and $n'\in N'$.
Set $V_{P'}=V/V(P')$ and we call this the Jacquet module of $V$ with respect to $P'$.
We write the natural projection $V\rightarrow V_{P'}$ by $j_{P'}$ and define the representation $\sigma_{P'}$ of $M'$ on $V_{P'}$ by
\[
\sigma_{P'}(m)j_{P'}(v)=\de_{P'}^{-1/2}(m')j_{P'}(\sigma(m')v), \hspace{10pt} v\in V,\ m'\in M'.
\]

Let $\th$ be the non-degenerate character of $N$ given by
\[
\th(u)=\psi_E\left(\sum_{i=1}^{n-1}u_{i, i+1}\right), \hspace{10pt} u=(u_{i, i+1})\in N.
\]
There is only one $T$-orbit of non-degenerate character of $N$.
Thus, we simply say a representation of $G$ is generic if it is $\th$-generic.

If we regard $G$ as a group over $E$, its $L$-group is given by ${}^LG_{/E}=\GL_n(\C)\times W_E$.
On the other hand, if we regard it as a group over $F$, its $L$-group is given by ${}^LG_{/F}=(\GL_n(\C)\times\GL_n(\C))\rtimes W_F$.
Here, the action of $W_F$ on $\GL_n(\C)\times\GL_n(\C)$ factors through $W_F/W_E\cong\mathrm{Gal}(E/F)$ and the action of $s$ is the permutation of the first and second components.
For an $L$-parameter $\phi\,\colon WD_E\rightarrow{}^LG_{/E}$, we define the corresponding $L$-parameter $\tilde{\phi}\,\colon WD_F\rightarrow{}^LG_{/F}$ by
\[\renewcommand{\arraystretch}{1.3}
\begin{array}{l}
\tilde{\phi}(x)=(\phi_1(x), \phi_1(sxs^{-1}))\rtimes\phi_2(x), \hspace{10pt} x\in WD_E \\
\tilde{\phi}((s, 1))=(1, \phi_1(s^2))\rtimes s.
\end{array}
\]
Here, each $\phi_i$ is the composition of $\phi$ with the $i$-th projection of ${}^LG_{/E}$.
The equivalence class of $\tilde{\phi}$ does not depend on the choice of $s$ and this construction provides a bijection between two sets of equivalence classes of $L$-parameters.
If we simply say an $L$-parameter of $G$, it means an $L$-parameter which takes its values in ${}^LG_{/E}$.

By Harris, Taylor \cite{harris2001geometry} and Henniart \cite{henniart2000preuve}, the local Langlands conjecture for $G$ was established.
Thus there is a bijection between the set of equivalence classes of $L$-parameters of $G$ and the set of irreducible representations of $G(F)$. 

We define a complex representation $\As^+$ of ${}^LG_{/F}$ on $\C^n\otimes\C^n$ as follows: for $a\otimes b\in\C^n\otimes\C^n$,
\[\renewcommand{\arraystretch}{1.3}
\begin{array}{l}
\As^+((g, h)\rtimes x)(a\otimes b)=(ga)\otimes(hb), \hspace{10pt} (g, h)\in\GL_n(\C)\times\GL_n(\C), \ x\in W_E \\
\As^+((1, 1)\rtimes s)(a\otimes b)=b\otimes a.
\end{array}
\]
Similarly, we define a representation $\As^-$ of ${}^LG_{/F}$ on the same space as follows: for $a\otimes b\in\C^n\otimes\C^n$,
\[\renewcommand{\arraystretch}{1.3}
\begin{array}{l}
\As^-((g, h)\rtimes x)(a\otimes b)=(ga)\otimes(hb), \hspace{10pt} (g, h)\in\GL_n(\C)\times\GL_n(\C), \ x\in W_E \\
\As^-((1, 1)\rtimes s)(a\otimes b)=-b\otimes a.
\end{array}
\]
In hornor of T. Asai, we call these representations \textit{Asai representations}.

According to \cite[Proposition 7.4]{gan2011symplectic}, if we identify $\C^n\otimes\C^n$ with the space $\M_n(\C)$ of $n$ by $n$ matrices, the stabilizer in $\As^{(-1)^{n-1}}$ of a vector corresponding to an invertible matrix is isomorphic to the $L$-group of the unitary group $\U_n$.
We say that an $L$-parameter $\tilde{\phi}\,\colon WD_F\rightarrow{}^LG_{/F}$ \textit{fixes a non-degenerate element of} $\As^{\pm}$ if the fixed point subspace of the image of $\As^{\pm}\circ\,\tilde{\phi}$ contains an invertible matrix.

\subsection{Even unitary groups and unstable base change lift}
From now on, we assume that $n=2m$ is even.
Let $G'=G'_m$ be the quasi-split even unitary group of rank $m$ defined by
\[
G'=\{g\in G\,;\, \t g^cJg=J\}.
\]
Here, 
\[
J=J_n=\left(\begin{array}{cc}
    &w \\
-w&    \\
\end{array}\right), \hspace{10pt} w=w_m=\left(\begin{array}{ccc}
   &          &1 \\
  &\idots&   \\
1&          &   \\
\end{array}\right)\in\GL_m.
\]
Let $B'$ denote the Borel subgroup of $G'$ consisting of upper triangular matrices.
Denote by $T'$ the maximal torus in $B'$ consisting of diagonal matrices and by $N'$ the unipotent radical of $B'$.

There are two $T'$-orbits of non-degenerate characters of $N'$.
We take representatives $\th'_{\ka}$ of these orbits given by
\[
\th'_{\ka}(u)=\psi_E\left(\sum_{i=1}^{m-1}u_{i, i+1}+\ka u_{m, m+1}\right), \hspace{10pt} u=(u_{i, j})\in N'.
\]
Here, $\ka$ runs over a set of representatives of $F^{\times}/\N_{E/F}(E^{\times})$ (cf. \cite[Proposition 12.1]{gan2011symplectic}).
We denote simply by $\th'$ the character corresponding to the element of $F^{\times}/\N_{E/F}(E^{\times})$.

The $L$-group of $G'$ is given by ${}^LG'=\GL_n(\C)\rtimes W_F$.
Here, the action of $W_F$ factors through $W_F/W_E\cong\rm{Gal}(E/F)$ and the action of $s$ is given by
\[
sgs^{-1}=J\t g^{-1}J^{-1}, \hspace{10pt} g\in\GL_n(\C).
\]
Take a character $\chi$ of $E^{\times}$ whose restriction to $F^{\times}$ coincides with $\omega$ and regard it as a character of $W_E$ as before.
Note that we have 
\[
\chi(sxs^{-1})=\chi(x)^{-1}
\]
for $x\in W_E$ and 
\[
\chi(s^2)=-1.
\]
We define the embedding of $L$-groups
\[
bc_{\chi}\,\colon{}^LG'\rightarrow{}^LG_{/F}
\]
as follows:
\[\renewcommand{\arraystretch}{1.5}
\begin{array}{ll}
bc_{\chi}(g\rtimes1)=(g,\ J\,\t g^{-1}J^{-1})\rtimes1, &g\in\GL_{n}(\C); \\
bc_{\chi}(\1_{n}\rtimes x)=(\chi(x)\cdot\1_{n},\ \chi(x)^{-1}\cdot\1_{n})\rtimes x, &x\in W_E;\\
bc_{\chi}(\1_{n}\rtimes s)=(-\1_n, \1_n)\rtimes s. &
\end{array}\]
For an irreducible representation $\pi$ of $G(F)$, let $\tilde{\phi}$ be the corresponding $L$-parameter.
We say that $\pi$ arises from \textit{unstable base change lift} of an $L$-parameter of $G'$ if there exists an $L$-parameter $\phi'$ of $G'$ which satisfies $\tilde{\phi}=bc_\chi\circ\phi'$.

One can define another embedding $bc_{\1}$ by 
\[\renewcommand{\arraystretch}{1.5}
\begin{array}{ll}
bc_{\1}(g\rtimes1)=(g,\ J\,\t g^{-1}J^{-1})\rtimes1, &g\in\GL_{n}(\C); \\
bc_{\1}(\1_{n}\rtimes x)=(\1_{n},\ \1_{n})\rtimes x, &x\in W_E;\\
bc_{\1}(\1_{n}\rtimes s)=(\1_n, \1_n)\rtimes s. &
\end{array}\]

The next result is well-known.
\begin{lem}\label{lem:bclift}
For an $L$-parameter $\phi$ of $G$, the following conditions are equivalent:
\begin{itemize}
\item[(1)] $\tilde{\phi}$ fixes a non-degenerate element of $\As^+$;
\item[(2)] there is an $L$-parameter $\phi'$ of $G'$ which satisfies $\tilde{\phi}= bc_{\chi}\circ\phi'$;
\item[(3)] $\ul{\phi}$ is conjugate-orthogonal.
\end{itemize}
Moreover, if $\phi$ satisfies these conditions, we have $L(s, \Ad\circ\,\phi')=L(s, \As^+\circ\,\tilde{\phi})$ with $\phi'$ as in (2).
\end{lem}

\begin{rem}
Here, $L(s, \Ad\circ\,\phi')$ and $L(s, \As^+\circ\,\tilde{\phi})$ are $L$-functions associated with representations of Weil-Deligne groups.
For their definition, see \cite[Section 2]{gross2010arithmetic}.
\end{rem}

\begin{proof}
Equivalence of (1) and (2) is \cite[Corollary 8.2]{gan2011symplectic}.
Equivalence of (1) and (3) is \cite[Proposition 7.5]{gan2011symplectic}.
See also \cite[Lemma 2.2.1]{mok2015endoscopic}.
The last assertion is \cite[Proposition 7.4]{gan2011symplectic}.

Note that we use the embedding $bc_{\chi}$ to associate a representation of $WD_E$ with an $L$-parameter of $G'$ instead of $bc_{\1}$.
This is why the conditions above are different from those in \cite{gan2011symplectic}.
\end{proof}

\subsection{$L$-packet}\label{subsection:L-packet}
In this subsection, we describe the desiderata for the local Langlands correspondence for $G'$.
For known results and precise conjectures, see \cite{mok2015endoscopic} and \cite{gan2011symplectic}.

For an $L$-parameter $\phi'$ of $G'$, set $\tilde{\phi}=bc_{\chi}\,\circ\phi'$.
Let $\phi$ be the $L$-parameter of $G$ corresponding to $\tilde{\phi}$ and $\ul{\phi}\,\colon WD_E\rightarrow\GL(M)$ be the representation of $WD_E$ attached to $\phi$.

For each $L$-parameter $\phi'$, it is conjectured that there exists a finite set $\Pi_{\phi'}$ of equivalence classes of irreducible representations of pure inner forms of $G'$ which satisfies following conditions:
\begin{itemize}
\item[(1)] The order of $\Pi_{\phi'}$ is equal to the number of irreducible representations of the finite group $A_M$, hence the order of $A_M$.
Moreover, a generic character $\th'$ of $N'$ determines a bijection
\[
J(\th')\,\colon \Pi_{\phi'}\rightarrow\Irr(A_M).
\]
Here, $\Irr(A_M)$ is the set of irreducible representations of $A_M$.

\item[(2)] The number of $\th'$-generic representations contained in $\Pi_{\phi'}$ is at most 1.
The finite set $\Pi_{\phi'}$ contains a $\th'$-generic representation if and only if $L(s, \Ad\circ\,\phi')$ is holomorphic at $s=1$.
This part is called the conjecture of Gross-Prasad and Rallis.
If this is the case, we say that $\phi'$ or $\Pi_{\phi'}$ is generic.
If $\Pi_{\phi'}$ is generic, the $\th'$-generic representation in $\Pi_{\phi'}$ corresponds to the trivial representation of $A_M$ via the map $J(\th')$.
In addition, the $\th'_{\tau}$-generic representation in $\Pi_{\phi'}$ corresponds to the quadratic character $\eta$ defined in section 1.1.
Here, $\tau$ denotes a representative of the nontrivial element of $F^{\times}/\N_{E/F}(E^{\times})$.
We say that $\phi'$ is \textit{generic with respect to any non-degenerate characters} if $\phi'$ is generic and $\eta$ is trivial.

\item[(3)] An element of $\Pi_{\phi'}$ corresponding to a character $\mu$ of $A_M$ via the map $J(\th')$ is a representation of $G'$ if and only if $\mu(-1)=+1$.
\end{itemize}
The finite set $\Pi_{\phi'}$ is called the $L$-\textit{packet} corresponding to $\phi'$.

\section{Quaternion algebra and symmetric subgroup}
\subsection{Definitions}
Denote the unique quaternion division algebra over $F$ by $D$ and we fix an embedding of $E$ to $D$.
Fix a representative $\tau$ of the nontrivial element of $F^{\times}/\N_{E/F}(E^{\times})$.
Then there is an element $\sqrt{\tau}$ in $D^{\times}$ which satisfies
\begin{itemize}
\item as a vector space over $E$, $D=E\oplus E\sqrt{\tau}$;
\item we have 
\[
(x_1+x_2\sqrt{\tau})(y_1+y_2\sqrt{\tau})=(x_1y_1+x_2y_2^c\tau)+(x_1y_2+x_2y_1^c)\sqrt{\tau}
\]
for any $x_1, x_2, y_1, y_2\in E$.
\end{itemize}

Define inductively an element $\ve_m$ of $G_m$ by
\[
\ve_1=\left(\begin{array}{cc}
  &\tau \\
1&
\end{array}\right), \hspace{10pt} \ve_m=\left(\begin{array}{cc}
\ve_1&    \\
        &\ve_{m-1}
\end{array}\right).
\]
Let $\sigma$ be the involution of $G$ given by 
\[
\sigma(g)=\ve_mg^c\ve_m^{-1}, \hspace{10pt} g\in G.
\]
The subgroup $H=H_m$ of fixed points of $\sigma$ is isomoprhic to $\GL_m(D)$.
For $m=1$ case, $H_1$ consists of matrices of the form
\[
\left(\begin{array}{cc}
a    &b\tau \\
b^c&a^c
\end{array}\right)
\]
with $a, b\in E$ satisfying $a^2-\tau b^2\neq0$.
Thus, an explicit isomorphism $H_1\cong D^{\times}$ is given by
\begin{equation}\label{eq:H_1}
\left(\begin{array}{cc}
a    &b\tau \\
b^c&a^c
\end{array}\right)\in H_1\mapsto a+b\sqrt{\tau}\in D^{\times}.
\end{equation}
In general, write an element $h$ of $H_m$ in the form of a 2 by 2 block matrix $(h_{i, j})_{1\leq i, j\leq m}$.
Each entry $h_{i, j}$ is of the form (\ref{eq:H_1}) and it corresponds to an element of $D$ similarly as above.
This observation provides an isomorphism $H_m\cong\GL_m(D)$.

Let $\la=(n_1, \ldots, n_r)$ be a partition of $n$.
Suppose each $n_i$ is even and write $n_i=2m_i$.
Then, $P_{\la}\cap H$ is a parabolic subgroup of $H$ which is, under the above isomorphism, sent to the parabolic subgroup of $\GL_m(D)$ consisting of matrices of the form
\[
\left(\begin{array}{ccc}
g_1&\ast     &\ast \\
      &\ddots&\ast \\
      &           &g_r
\end{array}\right)\in\GL_m(D), \hspace{10pt} g_i\in\GL_{m_i}(D).
\]
We call these parabolic subgroups standard.
Let $K_H$ be a maximal compact subgroup of $H$ which satisfies $H=(P_{\la}\cap H)K_H$ for any standard parabolic subgroups $P_{\la}\cap H$.

Let $x=x_m$ be the element of $\Sf_n$ defined by
\[
x(i)=\left\{\begin{array}{ll}
n-2i+1 &(1\leq i\leq m) \\
2i-n &(m+1\leq i\leq n).
\end{array}\right.
\]
Set $H'=H'_m=x^{-1}Hx$.
This is the subgroup of fixed points of the involution $\sigma'\,\colon g\mapsto \ve' g^c{\ve'}^{-1}$, where 
\[
\ve'=\ve'_m=x^{-1}\ve x=\left(\begin{array}{cc}
        &\tau w_m \\
w_m&
\end{array}\right).
\]
Hence, $H'$ consists of matrices of the form
\[
\left(\begin{array}{cc}
A                &\tau B\\
w_mB^cw_m&w_mA^cw_m
\end{array}\right)\in G, \hspace{10pt} A, B\in\M_m(E).
\]
The map $\left(\begin{array}{cc}
A                &\tau B\\
w_mB^cw_m&w_mA^cw_m
\end{array}\right)\mapsto A+\sqrt{\tau}B\in\GL_m(D)$ defines an isomorphism from $H'$ to $\GL_m(D)$.

\subsection{Double coset decomposition}
Let $\la=(n_1, \ldots, n_r)$ be a partition of $n$.
In this section, we recall some results on the double coset decomposition $H'\bs G/P_{\la}$ from \cite{matringe2017distinction}.

Let $\Sc(\la)$ be the set of $r$ by $r$ matrices $S=(s_{i, j})$ which satisfies following conditions:
\begin{itemize}
\item all entries are non-negative integers;
\item for any $1\leq i, j\leq r$, $s_{i, j}=s_{j, i}$, \textit{i.e.} $S=\t S$;
\item for any $1\leq i\leq r$, $s_{i, i}\in 2\Z_{\geq 0}$;
\item for any $1\leq i\leq r$, $\sum_{j=1}^rs_{i, j}=n_i$.
\end{itemize}
Take an element $S$ of $\Sc(\la)$.
For each $1\leq i\leq r$, set $t_i=s_{i, i}/2$ and $d_i=\sum_{j=i}^r s_{i, j}-t_i$.
Then, we obtain a partition $(d_1, \ldots, d_r)$ of $m$ and a partition $(t_i, s_{i, i+1}, s_{i, i+2}, \ldots, s_{i, r})$ of $d_i$ for each $1\leq i\leq r$.
Define an element $w_S$ as follows: for any $1\leq l\leq r$,
\begin{itemize}
\item[(w1)] if there are integers $i$ and $k$ satisfying $1\leq i\leq r$ and $1\leq k\leq t_i$ so that $l$ can be expressed as 
\[
l=d_1+\cdots+d_{i-1}+k,
\] we set 
\[
w_S(l)=(n_1+\cdots+n_{i-1})+s_{i, 1}+\cdots+s_{i, i-1}+k;
\]

\item[(w2)] if there are integers $i$ and $k$ satisfying $1\leq i<j\leq r$ and $1\leq k\leq s_{i, j}$ so that $l$ can be expressed as 
\[
l=(d_1+\cdots+d_{i-1})+(t_i+s_{i, i+1}+\cdots+s_{i, j-1})+k,
\]
we set 
\[
w_S(l)=(n_1+\cdots+n_{i-1})+(s_{i, 1}+\cdots+s_{i, j-1})+k;
\]

\item[(w3)] if there are integers $i$ and $k$ satisfying $1\leq i\leq r$ and $1\leq k\leq t_i$ so that $l$ can be expressed as 
\[
l=m+(d_r+\cdots+d_{i+1})+(s_{r, i}+s_{r-1, i}+\cdots+s_{i+1, i})+k,
\] we set 
\[
w_S(l)=(n_1+\cdots+n_{i-1})+s_{i, 1}+\cdots+s_{i, i-1}+t_i+k;
\]

\item[(w4)] if there are integers $i$ and $k$ satisfying $1\leq j<i\leq r$ and $1\leq k\leq s_{i, j}$ so that $l$ can be expressed as 
\[
l=m+(d_r+\cdots+d_{j+1})+(s_{r, j}+s_{r-1, j}+\cdots+s_{i+1, j})+k,
\]
we set 
\[
w_S(l)=(n_1+\cdots+n_{i-1})+(s_{i, 1}+\cdots+s_{i, j-1})+k.
\]
\end{itemize}

\begin{prop}[\cite{matringe2017distinction}, Proposition 3.1]
For a partition $\la$ of $n$, we have
\[
G=\coprod_{S\in\Sc(\la)}Pw_SH'.
\]
\end{prop}
An element $S=(s_{i, j})$ of $\Sc(\la)$ can be identified with a subpartition 
\[
(s_{1, 1}, s_{1, 2}, \ldots, s_{r, r-1}, s_{r, r})
\] of $\la$.
Here, 0-entries are ignored.
Set $\ve'_S=w_S\ve' w_S^{-1}$.
We define an involution $\sigma'_S$ of $G$ by $\sigma'_S\,\colon g\mapsto \ve'_S g^c{\ve'_S}^{-1}$.
Note that $P_S$ and $M_S$ is stable under $\sigma'_S$.
An element $m$ of $M_S$ is of the form $m=\d(m_{1, 1}, m_{1, 2}, \ldots, m_{r, r-1}, m_{r, r})$ with $m_{i, j}\in G_{s_{i, j}}$.
Of course components where $s_{i, j}=0$ are eliminated.
Then, $m$ is contained in $M_S^{\sigma'_S}$ if and only if the following conditions are satisfied:
\begin{itemize}
\item for any $1\leq i\leq r$, $m_{i, i}=\ve'_{t_i}m_{i, i}^c{\ve'_{t_i}}^{-1}$ \textit{i.e.} $m_{i, i}\in H'_{t_i}$;
\item for any $1\leq i\neq j\leq r$, $m_{i, j}=w_{s_{i, j}}m_{j, i}^cw_{s_{i, j}}^{-1}$.
\end{itemize} 
Hence, we get
\[
M_S^{\sigma'_S}\cong\prod_{i=1}^rH'_{t_i}\times\prod_{1\leq i<j\leq r}G_{s_{i, j}}.
\]

Set $P'_S=M\cap P_S$.
This is a parabolic subgroup of $M$.

\begin{prop}[\cite{matringe2017distinction}, Proposition 3.3]\label{prop:moduluscharacter}
We have $\de_{P_S^{\sigma'_S}}^2|_{M_S^{\sigma'_S}}=\de_{P'_S}\de_P|_{M_S^{\sigma'_S}}=\de_{P_S}|_{M_S^{\sigma'_S}}$.
\end{prop}

\subsection{Some results on $H'$-distinguished representations}
We record some results about $H'$-distinguished representations of $G$.
One can prove these results by almost the same arguments as in \cite{flicker1991distinguished} with minor modification.
So, we do not give a proof or only give a sketch.

\begin{prop}[See \cite{flicker1991distinguished}, Proposition 10]
For any $g\in G$, there are $x, y\in H'$ with $g^{-1}=x\sigma'(g)y$. 
\end{prop}

\begin{prop}[See \cite{flicker1991distinguished}, Proposition 11]
For an irreducible representation $\pi$ of $G$, the dimension of $\Hom_{H'}(\pi, \C)$ is at most 1.
\end{prop}

\begin{proof}
We sketch the proof.
For details, see \cite[Proposition 11]{flicker1991distinguished}.

Applying the theorem of Gel'fand and Kazhdan, and due to the above proposition, we are reduced to show the following claim:
\begin{quote}
For an irreducible representation $\pi$ of $G$, if $\Hom_{H'}(\pi, \C)$ is nonzero, then $\Hom_{H'}(\pi^{\vee}, \C)$ is nonzero.
\end{quote}
Let $l$ be an element of $\Hom_{H'}(\pi, \C)$.
We define a representation $\pi'$ of $G$ over the space of $\pi$ by
\[
\pi'(g)=\pi(w_n\t g^{-1}w_n^{-1}), \hspace{10pt} g\in G.
\]
Then, $\pi^{\vee}$ is isomorphic to $\pi'$.
Since $H'$ is stable under the involution $g\mapsto w_n\t g^{-1}w_n^{-1}$, $l$ is an element of $\Hom_{H'}(\pi', \C)$.
This proves the above claim.
\end{proof}

\begin{prop}[See \cite{flicker1991distinguished}, Proposition 12]
If an irreducible representation $\pi$ of $G$ is $H$-distinguished, then $\pi$ is conjugate self-dual, \textit{i.e.} $\pi^c\cong\pi^{\vee}$.
\end{prop}

\section{Zelevinsky classification}
In this section, we summarize some results on classification of irreducible representations of $G$ by Zelevinsky \cite{zelevinsky1980induced}.

First, we recall the classification of irreducible essentially square-integrable representations.

\begin{prop}
Let $l$ be a divisor of $n$ and set $\la=(n/l, \ldots, n/l)$ be a partition of $n$.
For an irreducible supercuspidal representation $\rho$ of $G_{n/l}$, the representation $\rho|\det|^{(1-l)/2}\times\cdots\times\rho|\det|^{(l-1)/2}$ of $G$ has a unique irreducible subrepresentation $\St_l(\rho)$.
This is essentially square-integrable.
Conversely, any irreducible essentially square-integrable representations of $G$ can be written in this form uniquely.
\end{prop}

With the notation of the above proposition, let $\ul{\phi}$ be the representation of $WD_E$ corresponding to $\rho$.
Since $\rho$ is supercuspidal, the restriction of $\ul{\phi}$ to $\SL_2(\C)$ is the trivial representation.
We denote the restriction of $\ul{\phi}$ to $W_E$ again by $\ul{\phi}$.
Let $\Sp(l)$ be the unique irreducible algebraic $l$-dimensional representation of $\SL_2(\C)$.
Then, the representation of $WD_E$ corresponding to $\St_l(\rho)$ is $\ul{\phi}\boxtimes\Sp(l)$.

A Jacquet module of $\St_l(\rho)$ is given as follows:

\begin{lem}\label{Jacquet}
Let $l$ be a divisor of $n$ and $\rho$ be an irreducible supercuspidal representation of $G_{n/l}$.
Set $\pi=\St_l(\rho)$ and denote by $V$ the space of $\pi$.
Let $\la=(n_1, \ldots, n_r)$ be a partition of $n$ and set $P=P_{\la}$.
Then, $V_P\neq0$ ifa nd only if there is a non-negative integer $k_i$ for each $1\leq i\leq r$ satisfying $n_i=k_i\cdot n/l$.
If this is the case, we have
\begin{equation}\label{Jac}
\pi_P\cong\St_{k_1}(\rho|\det|^{(l-k_1)/2})\boxtimes\St_{k_2}(\rho|\det|^{-k_1+(l-k_2)/2})\boxtimes\cdots\boxtimes\St_{k_r}(\rho|\det|^{(k_r-l)/2}).
\end{equation}
\end{lem}

By Langlands classification, all irreducible representations of reductive groups appear in the composition factors of parabolically induced representations from essentially square-integrable representations.
In the case at hand, more precise classification is known.

\begin{thm}\label{thm:Zelevinskyclassification}
Let $\la=(n_1, \ldots, n_r)$ be a partition of $n$ and $\Delta_i$ be an irreducible essentially square-integrable representation of $G_{n_i}$.
We write the absolute value of the central character of each $\Delta_i$ by $|\det|^{t_i}$ with real number $t_i$.
Reordering $\Delta_i$'s so that we have
\begin{equation}\label{eq:Langlandsquot}
t_1/n_1\geq  t_2/n_2\geq\cdots\geq t_r/n_r,
\end{equation}
the representation $\Delta_1\times\cdots\times\Delta_r$ has a unique irreducible quotient that we denote as $\Delta_1\boxplus\cdots\boxplus\Delta_r$.
This representation is independent of the reordering which satisfies (\ref{eq:Langlandsquot}) and any irreducible representations of $G$ can be expressed in this form.
Moreover, this expression is unique up to permutation.
\end{thm}

Generic representations are classified in terms of the above classification as follows:

\begin{prop}
Let $\pi=\St_{l_1}(\rho_1)\boxplus\cdots\boxplus\St_{l_r}(\rho_r)$ be an irreducible representation of $G$.
Here,
\begin{itemize}
\item $\la=(n_1, \ldots, n_r)$ is a partition of $n$;
\item $l_i$ is a divisor of $n_i$;
\item $\rho_i$ is an irreducible supercuspidal representation of $G_{n_i/l_i}$.
\end{itemize}
Then, $\pi$ is generic if and only if there are no $i$, $j$ with $1\leq i\neq j\leq r$ and a positive integer $d$ satisfying the two conditions ($\heartsuit$1)--($\heartsuit$2):
\begin{itemize}\setlength{\itemsep}{3pt}
\item[($\heartsuit$1)] $\max\{1, l_i-l_j+1\}\leq d\leq l_i$ and $n_i/l_i=n_j/l_j$;
\item[($\heartsuit$2)] $\rho_i\cong\rho_j\otimes|\det|^{d+(l_j-l_i)/2}$.
\end{itemize}
If this is the case, $\pi\cong\St_{l_{w(1)}}(\rho_{w(1)})\times\cdots\times\St_{l_{w(r)}}(\rho_{w(r)})$ for any permutation $w\in\Sf_r$.
\end{prop}

\begin{rem}\label{rem:genericrep}
Let $\pi$ be an irreducible representation of $G$ and $\ul{\phi}$ be the corresponding representation of $WD_E$.
The above characterization of generic representations can be restated as follows: $\pi$ is generic if and only if $L(s, \Ad\circ\,\ul{\phi})$ is holomorphic at $s=1$.
Suppose that $\pi$ is generic and there is an $L$-parameter $\phi'$ of $G'$ such that $\tilde{\phi}=bc_{\chi}\circ\,\phi'$.
Combining the above characterization of generic representations with its analog for $G'$ (\S \ref{subsection:L-packet} (2)), one can see that $\Pi_{\phi'}$ is generic.
Another way to prove this is to use the equation (\cite[Theorem 5.3]{matringe2010distinguished}) 
\[
L(s, \Ad\circ\,\phi')=L(s, \As^+\circ\,\tilde{\phi})=L(s, \pi, \As^+)
\]
and the fact that $L(s, \pi, \As^+)$ is holomorphic at $s=1$ (\cite[Proposition 7.1]{anandavardhanan2017test}).
\end{rem}

\section{Classification of distinguished representations}
First, we recall the classification of $\GL_n(F)$-distinguished generic representations of $G$ by Matringe.
Note that this result also holds for odd $n$.

\begin{thm}[\cite{matringe2010distinguished}, Theorem 5.2]\label{thm:classificationGL-distinguished}
Let $\pi$ be an irreducible generic representation of $G$ and write $\pi=\Delta_1\boxplus\cdots\boxplus\Delta_r$ as in Theorem \ref{thm:Zelevinskyclassification}.
Here, $\la=(n_1, \ldots, n_r)$ is a partition of $n$ and each $\Delta_i$ is an irreducible essentially square-integrable representation of $G_{n_i}$.
Then, $\pi$ is $\GL_n(F)$-distinguished if and only if there are an integer $1\leq k\leq r/2$ and a reordering of $\Delta_i$'s satisfying
\begin{itemize}
\item[(1)] for all $1\leq i\leq k$, $\Delta_{2i}^c\cong\Delta_{2i-1}^{\vee}$;
\item[(2)] for all $2k+1\leq i\leq r$, $\Delta_i$ is $\GL_{n_i}(F)$-distinguished.
\end{itemize}
\end{thm}

We have a similar classification for $H$-distinguished generic representations.

\begin{thm}\label{thm:classificationH-distinguished}
Let $\pi$ be an irreducible generic representation of $G$ and write $\pi=\Delta_1\boxplus\cdots\boxplus\Delta_r$ as in Theorem \ref{thm:Zelevinskyclassification}.
Here, $\la=(n_1, \ldots, n_r)$ is a partition of $n$ and each $\Delta_i$ is an irreducible essentially square-integrable representation of $G_{n_i}$.
Then, $\pi$ is $H$-distinguished if and only if there are an integer $1\leq k\leq r/2$ and a reordering of $\Delta_i$'s satisfying
\begin{itemize}
\item[(1)] for all $1\leq i\leq k$, $\Delta_{2i}^c\cong\Delta_{2i-1}^{\vee}$;
\item[(2)] for all $2k+1\leq i\leq r$, $n_i=2m_i$ is even and $\Delta_i$ is $H_{m_i}$-distinguished.
\end{itemize}
\end{thm}

Comparing these two classification results, we obtain the next corollary.

\begin{cor}
An irreducible generic representation of $G$ is $\GL_n(F)$-distinguished if it is $H$-distinguished.
\end{cor}

The next lemma is the key to the proof of Theorem \ref{thm:classificationH-distinguished}.

\begin{lem}
Let $\la=(n_1, \ldots, n_r)$ be a partition of $n$ and $\Delta_i$ an irreducible essentially square-integrable representation of $G_{n_i}$ for each $1\leq i\leq r$.
Set $\Delta=\Delta_1\boxtimes\cdots\boxtimes\Delta_r$ and regard is as a representation of $M_{\la}$.
Suppose that the Jacquet module $\Delta_{P'_S}$ is $M_{S}^{\sigma'_S}$-distinguished for some $S=(s_{i, j})\in\Sc(\la)$. 
Then, there are an integer $1\leq k\leq r/2$ and a reordering of $\Delta_i$'s satisfying (1) and (2) in Theorem \ref{thm:classificationH-distinguished}.
\end{lem}

\begin{proof}
The proof is the same as \cite[Lemma 5.5]{matringe2010distinguished}.
We write each $\Delta_i$ in the form of $\St_{l_i}(\rho_i)$.
Here, $l_i$ is a divisor of $n_i$ and $\rho_i$ is an irreducible supercuspidal representation of $G_{n_i/l_i}$.
We may assume 
\begin{equation}\label{assumption}
l_1\leq l_2\leq \cdots\leq l_r.
\end{equation}
Since we have $\Delta_{P'_S}\neq\{0\}$, there exists a non-negative integer $k_{i, j}$ satisfying $s_{i, j}=k_{i, j}\cdot n_i/l_i$ for each $1\leq i, j\leq r$.
Hence we have 
\[
\Delta_{P'_S}\cong\bigotimes_{1\leq i, j\leq r}\Delta_{i, j}, \hspace{10pt} \Delta_{i, j}=\rm{St}_{k_{i, j}}(\rho_i|\det|^{-k_{i, 1}-\cdots-k_{i, j-1}+(l_i-k_{i, j})/2})
\]
and
\begin{itemize}
\item for any $1\leq i\leq r$, $\Delta_{i, i}$ is $H'_{t_i}$-distinguished;
\item for any $1\leq i, j\leq r$, $\Delta_{j, i}^c\cong\Delta_{i, j}^{\vee}$.
\end{itemize}

By induction on $r$, we show that there are an integer $1\leq k\leq r/2$ and reordering of $\Delta_i$'s which satisfy the conditions (1) and (2) in Theorem \ref{thm:classificationH-distinguished}.
Set $i_0=\min\{1\leq i\leq r\,;\, s_{1, i}\neq 0\}$.

\ul{Case 1.}\hspace{3pt} If $i_0=1$, $\Delta_{1, 1}$ is conjugate self-dual since it is $H'_{t_1}$-distinguished.
Thus we get $\rho_1^{\vee}\cong\rho_1^c|\det|^{l_1-k_{1, 1}}$.
Therefore, $\Delta_1^{\vee}$ is isomorphic to $\St_{l_1}(\rho_1^c|\det|^{l_1-k_{1, 1}})$.
Take $1\leq j\leq r$ with $\Delta_1^{\vee}\cong\Delta_j^c$.
Then, $n_1=n_j$, $l_1=l_j$ and $\rho_1$ is isomorphic to $\rho_j|\det|^{k_{1, 1}-l_1}$.
From ($\heartsuit$1) and($\heartsuit$2), one can see that $k_{1, 1}=l_1$, \textit{i.e.} $s_{1, 1}=n_1$ and $s_{1, j}=0$ for any $1<j\leq r$.
Hence we have $\Delta_1=\Delta_{1, 1}$ and this is $H'_{m_1}$-distinguished.
By the induction hypothesis, there is a reordering of $\Delta_2, \ldots, \Delta_r$ which satisfies two conditions of Theorem \ref{thm:classificationH-distinguished} for suitable $1\leq k\leq (r-1)/2$.

\ul{Case 2.}\hspace{3pt} If $i_0>1$, note that we have $k_{1, j}=0$ for any $1\leq j\leq i_0$.
Since $\Delta_{1, i_0}^c\cong\St_{k_{1, i_0}}(\rho_1^c|\det|^{(l_1-k_{1, i_0})/2})$ is isomorphic to $\Delta_{i_0, 1}^{\vee}\cong\St_{k_{i_0, 1}}(\rho_{i_0}^{\vee}|\det|^{(k_{i_0, 1}-l_{i_0})/2})$, we see that $n_1/l_1=n_{i_0}/l_{i_0}$ and $\rho_{i_0}^{\vee}\cong\rho_1^c|\det|^{(l_1+l_{i_0})/2-k_{1, i_0}}$.
Therefore, we obtain
\[
\Delta_{i_0}^{\vee}\cong\rm{St}_{l_{i_0}}(\rho_1^c|\det|^{(l_1+l_{i_0})/2-k_{1, i_0}}).
\]
Take $1\leq j\leq r$ so that $\Delta_{i_0}^{\vee}\cong\Delta_j^c$.
Then we have $n_j/l_j=n_{i_0}/l_{i_0}=n_1/l_1$ and $\rho_1\cong\rho_j|\det|^{k_{1, i_0}-(l_1+l_{i_0})/2}$.
Moreover, $n_j=n_{i_0}$ implies $l_j=l_{i_0}$.
Since we have $l_j-k_{1, i_0}<l_j$, the conditions ($\heartsuit$1), ($\heartsuit$2) and the assumption (\ref{assumption}) imply 
\[
0\geq l_j-k_{1, i_0}=l_{i_0}-k_{1, i_0}\geq l_1-k_{1, i_0}.
\]
Since $l_1\geq k_{1, i_0}$, we obtain $l_1=k_{1, i_0}$, \textit{i.e.} $s_{1, i_0}=n_1$ and for any $1\leq j\neq i_0\leq r$, we see that $s_{1, j}=0$.
Hence $\Delta_{1}=\Delta_{1, i_0}$.

We have two cases: $\Delta_1$ is conjugate self-dual (Case 2-a) or $\Delta_1$ is not conjugate self-dual (Case 2-b).

\ul{Case 2-a.}\hspace{3pt} If $\Delta_1=\Delta_{1, i_0}$ is conjugate self-dual, 
$\Delta_{1, i_0}^{\vee}\cong\Delta_{i_0, 1}^c$ is conjugate self-dual.
Hence we get $\rho_{i_0}^{\vee}\cong\rho_{i_0}^c|\det|^{l_{i_0}-l_1}$.
Thus, $\Delta_{i_0}^{\vee}$ is isomorphic to $\rm{St}_{l_{i_0}}(\rho_{i_0}^c|\det|^{l_{i_0}-l_1})$.
Take $1\leq j\leq r$ so that $\Delta_{i_0}^{\vee}$ is isomorphic to $\Delta_j^c$.
Then we have $n_{i_0}/l_{i_0}=n_j/l_j$ and $\rho_{j}\cong\rho_{i_0}|\det|^{l_{i_0}-l_1}=\rho_{i_0}|\det|^{l_j-l_1}$.
The conditions ($\heartsuit$1) and ($\heartsuit$2) imply $l_j=l_1$.
Hence we have $l_{i_0}=l_1=k_{1, i_0}=k_{i_0, 1}$, \textit{i.e.} $n_{i_0}=s_{i_0, 1}$ and $s_{i_0, j}=0$ for any $1<j\leq r$.
Therefore, one can see that $\Delta_{i_0}=\Delta_{i_0, 1}\cong\Delta_{1, i_0}^{\vee}\cong\Delta_1$.
By the induction hypothesis, there is a reordering of $\Delta_2, \ldots, \Delta_{i_0-1}, \Delta_{i_0+1}, \ldots, \Delta_r$ which satisfies two conditions of Theorem \ref{thm:classificationH-distinguished} for some $1\leq k\leq r/2-1$.

\ul{Case 2-b.}\hspace{3pt} If $\Delta_1=\Delta_{1, i_0}$ is not conjugate self-dual, $\Delta_2^c\cong\Delta_1^{\vee}$.
Note that $n_1=n_2$ and $l_1=l_2$.
Set $j_0=\min\{1\leq j\leq r\,;\, s_{2, j}\neq0\}$.
If $j_0=2$, one can obtain $\Delta_2=\Delta_{2, 2}$ by the similar arguments as in the Case 1.
Since $\Delta_2$ is $H'_{m_2}$-distinguished, by the induction hypothesis, there is a reordering of $\Delta_1, \Delta_3, \ldots, \Delta_r$ which satisfies two conditions of Theorem \ref{thm:classificationH-distinguished} for some $1\leq k\leq (r-1)/2$.
Hereafter, we assume that $j_0\neq2$.

Suppose that $j_0>2$.
Since $\Delta_{2, j_0}^c=\rm{St}_{k_{2, j_0}}(\rho_2^c|\det|^{(l_2-k_{2, j_0})/2})$ is isomorphic to $\Delta_{j_0, 2}^{\vee}\cong\rm{St}_{k_{2, j_0}}(\rho_{j_0}^{\vee}|\det|^{k_{j_0, 1}+(k_{j_0, 2}-l_{j_0})/2})$, we obtain 
\[
\rho_{j_0}^{\vee}\cong\rho_2^c|\det|^{-k_{j_0, 1}-k_{j_0, 2}+(l_{j_0}+l_2)/2}.
\]

If $j_0=i_0$, we have $\rho_{i_0}^{\vee}\cong\rho_1^c|\det|^{(l_{i_0}-l_1)/2}$ and $\rho_{i_0}^{\vee}\cong\rho_2^c|\det|^{-k_{j_0, 2}+(l_{i_0}-l_1)/2}$ as we see above.
Then we get $\rho_2\cong\rho_1|\det|^{k_{j_0, 2}}$.
From ($\heartsuit$1) and ($\heartsuit$2), we see that $k_{j_0, 2}=0$, which contradicts the choice of $j_0$.
Thus, $j_0\neq i_0$.

Take $1\leq j\leq r$ so that $\Delta_{j_0}^{\vee}\cong\Delta_j^c$.
Since we have 
\[
\Delta_{j_0}^{\vee}\cong\rm{St}_{l_{j_0}}(\rho_2^c|\det|^{-k_{j_0, 2}+(l_{j_0}+l_2)/2}),
\]
we see that $n_j/l_j=n_{j_0}/l_{j_0}=n_2/l_2$ and $\rho_j=\rho_2|\det|^{-k_{j_0, 2}+(l_j+l_2)/2}$.
From the conditions ($\heartsuit$1), ($\heartsuit$2) and the assumption (\ref{assumption}), one can see that $l_2=k_{j_0, 2}$ as $l_j\geq l_j-k_{j_0, 2}\geq l_j-l_2\geq0$.
Hence $s_{2, j_0}=n_2$ and $s_{2, j}=0$ for any $1\leq j\neq j_0\leq r$.
Therefore we obtain $\Delta_2=\Delta_{2, j_0}$.

If $j_0=1$, we have $s_{1, 2}=s_{2, 1}\neq0$ and hence $i_0=2$.
As we have seen, $\Delta_1=\Delta_{1, 2}$.
This implies $\Delta_2^c\cong\Delta_{1, 2}^{\vee}\cong\Delta_{2, 1}^c$.
Hence we get $\Delta_2=\Delta_{2, 1}$.

In any case, we see that $s_{2, j}=0$ for all $1\leq j\neq j_0\leq r$, \textit{i.e.} $k_{2, j_0}=l_2=l_1$, $n_2=s_{2, j_0}$.
Since $\Delta_1^c\cong\Delta_{i_0, 1}^{\vee}\cong\rm{St}_{l_1}(\rho_{i_0}^{\vee}|\det|^{(l_1-l_{i_0})/2})$ is isomorphic to $\Delta_2^{\vee}\cong\Delta_{j_0, 2}^c\cong\rm{St}_{l_2}(\rho_{j_0}^c|\det|^{(l_{j_0}-l_2)/2})$, we get $\rho_{i_0}^{\vee}\cong\rho_{j_0}^c|\det|^{(l_{j_0}+l_{i_0})/2-l_1}$.
Take $1\leq j'\leq r$ so that $\Delta_{i_0}^\vee\cong\Delta_{j'}^c$.
Then we have $n_{i_0}/l_{i_0}=n_{j'}/l_{j'}$ and $\rho_{j'}\cong\rho_{j_0}|\det|^{(l_{j_0}+l_{j'})/2-l_1}$.
If $l_1<l_{j_0}$, ($\heartsuit$1) and ($\heartsuit$2) imply $l_{j'}=l_1$.
If this is the case, 
\[
\Delta_{i_0}^\vee\cong\Delta_{j'}^c\cong\St_{l_2}(\rho_{j_0}^c|\det|^{(l_{j_0}-l_2)/2})\cong\Delta_2^\vee.
\]
Hence we get $\Delta_{i_0}\cong\Delta_2$.
By the induction hypothesis, there is a reordering of $\Delta_2, \ldots, \Delta_{i_0-1}, \Delta_{i_0+1}, \ldots, \Delta_r$ which satisfies two conditions of Theorem \ref{thm:classificationH-distinguished} for some $1\leq k\leq (r-1)/2$.
If $l_1=l_{j_0}$, $\Delta_{j_0}\cong\Delta_1$ since we have $\rho_{j_0}^c\cong\rho_{i_0}^\vee|\det|^{(l_1-l_{i_0})/2}\cong\rho_1^c$.
By the induction hypothesis, there is a reordering of $\Delta_1, \Delta_3, \ldots, \Delta_{j_0-1},\Delta_{j_0+1},\ldots, \Delta_r$ which satisfies two conditions of Theorem \ref{thm:classificationH-distinguished} for some $1\leq k\leq (r-1)/2$.

This completes the proof of the lemma.
\end{proof}

\begin{proof}[Proof of Theorem \ref{thm:classificationH-distinguished}]
Suppose that $\pi$ is $H'$-distinguished.
We simply write $P_{\la}$ by $P$.
There is a total order $\geq$ on the set $W_{\la}=\{w_S\,;\, S\in\Sc(\la)\}$ which satisfies following condition:
\begin{quote}
For any $w\in W_{\la}$, 
\[
G_{\geq w}=\coprod_{\substack{w'\in W_{\la}\\ w'\geq w}}Pw'H'
\]
is open in $G$ and $PwH'$ is closed in $G_{\geq w}$.
\end{quote}

Let $W_i$ be the space of $\Delta_i$ for each $1\leq i\leq r$ and set $(\Delta, W)=(\Delta_1\boxtimes\cdots\boxtimes\Delta_r, W_1\otimes\cdots\otimes W_r)$.
For each $w\in W_{\la}$, let $\mF_w$ be the $H'$-submodule of $I_P^GW$ consisting of elements whose support is contained in $G_{\geq w}$.
Set $\mF_{w^+}=\sum_{w'>w}\mF_w$.
The map $\mF_w\rightarrow\ind_{w^{-1}Pw\cap H'}^{H'}(w^{-1}(\de_P^{-1/2}\Delta))$ given by sending $f\in\mF_w$ to the function $h\mapsto f(wh)$ on $H'$ induces an isomorphism of $H'$-modules
\[
\mF_w/\mF_{w^+}\xrightarrow{\sim}\ind_{w^{-1}Pw\cap H'}^{H'}(w^{-1}(\de_P^{1/2}\Delta)).
\]

Take a divisor $l_i$ of $n_i$ and an irreducible supercuspidal representation $\rho_i$ of $G_{n_i/l_i}$ so that we have $\Delta_i\cong\St_{l_i}(\rho_i)$. 
Since $\pi$ is $H'$-distinguished, there is a $w\in W_{\la}$ such that $\ind_{w^{-1}Pw\cap H'}^{H'}(w^{-1}(\de_P^{1/2}\Delta))$ has a nonzero $H'$-invariant linear form.
Let $S$ be the element of $\Sc(\la)$ with $w=w_S$.
Then, we have $P\cap wH'w^{-1}=P_S^{\sigma'_S}$.
By Frobenius reciprocity, we obtain
\[\renewcommand{\arraystretch}{1.3}
\begin{array}{rl}
\Hom_{H'}(\rm{ind}_{w^{-1}Pw\cap H'}^{H'}(w^{-1}(\de_P^{1/2}\Delta)), \1)&\cong\Hom_{wH'w^{-1}}(\rm{ind}_{P\cap wH'w^{-1}}^{wH'w^{-1}}(\de_P^{1/2}\Delta), \1) \\
&\cong\Hom_{P_S^{\sigma'_S}}(\de_{P_S^{\sigma'_S}}^{-1}\de_P^{1/2}\Delta, \1).
\end{array}\]
By Proposition \ref{prop:moduluscharacter}, we see that the restrictions of $\de_{P_S^{\sigma'_S}}^{-1}\de_P^{1/2}$ and $\de_{P'_S}^{-1/2}$ to $P_S^{\sigma'_S}$ are equal.
Hence we get $\Hom_{P_S^{\sigma'_S}}(\de_{P'_S}^{-1/2}\Delta, \1)\neq\{0\}$.
Since we have $M\cap N_S=N'_S\subset N_S^{\sigma'_S}N$, the Jacquet module $\Delta_{P'_S}$ is $M_S^{\sigma'_S}$-distinguished.
By the above lemma, there are an integer $1\leq k\leq r/2$ and a reordering of $\Delta_i$'s which satisfy the two conditions in the theorem.

Conversely, suppose that there are an integer $1\leq k\leq r/2$ and a reordering of $\Delta_i$'s  which satisfy the two conditions in the theorem.
For each $2k+1\leq i\leq r$, take a nonzero element $\xi_i$ of $\Hom_{H_{m_i}}(\Delta_i, \1)$.

For $1\leq i\leq k$, let $Q_i$ be the standard parabolic subgroup of $G_{2n_{2i-1}}$ corresponding to the partition $(n_{2i-1}, n_{2i-1})$.
Let $\tau_{i, s}=\Delta_{2i-1}|\det|^{s}\times\Delta_{2i}|\det|^{-s}$ be a representation of $G_{2n_{2i-1}}$ with a parameter $s\in\C$.
Since $\Delta_{2i}^c$ is isomorphic to $\Delta_{2i-1}^{\vee}$, there is a nonzero $G_{n_{2i-1}}$-invariant linear form $\ga_i$ on $\Delta_{2i-1}\otimes\Delta_{2i}^c$.
Take a flat section $v_{i, s}$ of $\tau_{i, s}$ and set $v_i=v_{i, 0}$.
Then, by \cite[Theorem 2.8, Theorem 2.26]{blanc2008vecteurs}, the integral
\[
\xi_i(v, s)=\int_{L\bs H'_{n_{2i-1}}}\ga_i(v_{i, s}(h)x_{n_{2i-1}})\,dh
\]
converges absolutely for $\rm{Re}(s)\gg0$ and has meromorphic continuation to whole $s$-plane.
Here, 
\[
L=H'_{n_{2i-1}}\cap Q_i=\{\d(A, A^c)\mid A\in\GL_{n_{2i-1}}(E)\}.
\]
Hence, this integral defines a nonzero element of $\rm{Hom}_{H_{n_{2i-1}}}(\tau_{i, s}, \1)$.
Let $\xi_{i}(v)$ be the leading coefficient of its Laurent expansion at $s=0$, the map $v\mapsto\xi_i(v)$ provides a nonzero element of $\Hom_{H_{n_{2i-1}}}(\tau_i, \1)$.

Let $\tilde{P}=P_{\tilde{\la}}$ be the standard parabolic subgroup of $G$ corresponding to the partition $\tilde{\la}=(2n_1, 2n_3, \ldots, 2n_{2k-1}, n_{2k+1}, \ldots, n_r)$ of $n$.
We define the representation $(\tilde{\Delta}, \tilde{W})$ of $\tilde{M}=M_{\tilde{\la}}$ by 
\[
\tilde{\Delta}=\tau_1\boxtimes\cdots\tau_{k}\boxtimes\Delta_{2k+1}\boxtimes\cdots\boxtimes\Delta_r, 
\]
so that we have $\pi=\rm{Ind}_{\tilde{P}}^G(\tilde{\Delta})$. 
Let $\tilde{\xi}$ be the element of $\Hom_{H\cap\tilde{M}}(\tilde{\Delta}, \1)$ given by 
\[
\tilde{\xi}=\xi_1\boxtimes\cdots\boxtimes\xi_k\boxtimes\xi_{2k+1}\boxtimes\cdots\boxtimes\xi_r.
\]
Recall that we have $(\tilde{P}\cap H)K_{H}=H$ since $\tilde{P}\cap H$ is a standard parabolic subgroup of $H$.
For an element $\phi$ of $\pi=\rm{Ind}_{\tilde{P}}^G(\tilde{\Delta})$, set
\[
\xi(\phi)=\int_{K_{H}}\tilde{\xi}(\phi(k))\,dk.
\]
Then, $\xi$ defines a nonzero element of $\Hom_{H}(\pi, \1)$.
Therefore, $\pi$ is $H$-distinguished.
\end{proof}

\section{Main theorem}
\begin{lem}
Let $\Delta$ be an irreducible essentially square-integrable representation of $G$.
Then, following conditions are equivalent:
\begin{itemize}
\item[(1)] $\Delta$ is $\GL_m(D)$-distinguished;
\item[(2)] $\Delta$ is $\GL_n(F)$-distinguished;
\item[(3)] $\Delta$ arises from unstable base change lift;
\item[(4)] $\Delta$ arises from unstable base change lift of an $L$-parameter of $G'$ which is generic with respect to any non-degenerate characters.
\end{itemize}
\end{lem}

\begin{proof}
Equivalence of (1) and (2) is \cite[Theorem 1]{beuzart2017distinguished}.
Let us prove equivalence of (2) and (3).
By \cite[Theorem 3.7]{matringe2009conjectures}, $\Delta$ is $\GL_n(F)$-distinguished if and only if $L(s, \Delta, \As^+)$ has a pole at $s=0$.
Take a divisor $l$ of $n$ and an irreducible supercuspidal representation $\rho$ of $G_{n/l}$ so that $\Delta=\St_l(\rho)$.
Let $\ul{\phi_{\rho}}$ (resp. $\ul{\phi}$) be the representation of $WD_E$ associated with $\rho$ (resp. $\Delta$).
We write the restriction of $\ul{\phi_{\rho}}$ to $W_E$ by the same symbol.
Then, we have $\ul{\phi}=\ul{\phi_{\rho}}\boxtimes\Sp(l)$.
By \cite[Proposition 4.1]{matringe2009conjectures}, we have
\[
L(s, \Delta, \As^+)=\prod_{k=0}^{l-1}L(s+k, \omega^{l-k-1}\otimes\rho, \As^+).
\]
Since $L(s, \omega\otimes\rho, \As^+)=L(s, \rho, \As^-)$, $L(s, \Delta, \As^+)$ has a pole at$s=0$ if and only if one of the following conditions holds:
\begin{itemize}
\item $l$ is odd and $\wt{\phi_{\rho}}$ fixes a non-degenerate vector of $\As^+$
\item $l$ is even and $\wt{\phi_{\rho}}$ fixes a non-degenerate vector of $\As^-$.
\end{itemize}
When $l$ is odd, $\Sp(l)$ is orthogonal and when $l$ is even, $\Sp(l)$ is symplectic.
Besides, $\wt{\phi_{\rho}}$ fixes a non-degenerate vector of $\As^+$ (resp. $\As^-$) if and only if $\ul{\phi_{\rho}}$ is conjugate-orthogonal (resp. conjugate-symplectic), see \cite[Proposition 7.5]{gan2011symplectic}.
Hence, $\Delta$ is $\GL_n(F)$-distinguished if and only if $\ul{\phi}$ is conjugate-orthogonal.
By Lemma \ref{lem:bclift}, we obtain equivalence of (2) and (3).

Finally, we show that (3) implies (4).
Take an $L$-parameter $\phi'$ of $G'$ with $\tilde{\phi}=bc_{\chi}\circ\,\phi'$.
By Remark \ref{rem:genericrep}, $\Pi_{\phi'}$ is generic.
Since $n$ is even, the quadratic character $\eta$ of $A_{\phi}$ is trivial.
Hence $\phi'$ is generic with respect to any non-degenerate characters.
\end{proof}

\begin{rem}\label{rem:sq-int}
By a similar argument as the proof of the above lemma, one can see that following conditions for an irreducible essentially square-integrable representation $\Delta$ of $G$ are equivalent:
\begin{itemize}
\item $\Delta$ is $\GL_n(F)$-distinguished (resp. $(\GL_n(F)), \omega$-distinguished);
\item the representation of $WD_E$ corresponding to $\ul{\phi}$ is conjugate-orthogonal (resp. conjugate-symplectic).
\end{itemize}
These equivalence hold even when $n$ is odd.
\end{rem}

\begin{thm}\label{thm:main}
Let $\pi$ be an irreducible generic representation of $G$.
Consider the following statements:
\begin{itemize}
\item[(A)] $\pi$ arises from unstable base change lift of an $L$-parameter of $G'$ which is generic with respect to any non-degenerate characters;
\item[(B)] $\pi$ is $H$-distinguished.
\end{itemize}
Then, (A) implies (B).
\end{thm}

\begin{proof}
Suppose that (A) holds.
Let $\la=(n_1, \ldots, n_r)$ be a partition of $n$ and $\Delta_i$ an irreducible essentially square-integrable representation of $G_{n_i}$ with $\pi=\Delta_1\times\cdots\times\Delta_r$.
Denote by $\ul{\phi}$ the representation of $WD_E$ corresponding to $\pi$.
By Lemma \ref{lem:bclift}, $\ul{\phi}$ is conjugate-orthogonal.
Hence, there is an integer $1\leq k'\leq k\leq r/2$ and a reordering of $\Delta_i$'s satisfying
\begin{itemize}
\item[(1)] for all $1\leq i\leq k'$, $\Delta_{2i-1}$ is not conjugate self-dual and $\Delta_{2i}^c\cong\Delta_{2i-1}^{\vee}$;
\item[(2)] for all $k'+1\leq i\leq k$, $\Delta_{2i-1}$ is $(\GL_{n_{2i-1}}(F), \omega)$-distinguished and $\Delta_{2i-1}\cong\Delta_{2i}$;
\item[(3)] for all $2k+1\leq i\leq r$, $\Delta_i$ is $\GL_{n_i}(F)$-distinguished.
\end{itemize}
Note that $\Delta_i$'s appearing in (1) (resp. (2), (3)) correspond to the irreducible factors of $\ul{\phi}$ which is not conjugate self-dual (resp. conjugate-symplectic, conjugate-orthogonal), cf.\,Remark \ref{rem:sq-int}.
Since the quadratic character $\eta$ of $A_{\phi}$ is trivial, $n_i=2m_i$ is even for all $2k+1\leq i\leq r$.
Hence by \cite[Theorem 1]{beuzart2017distinguished}, each $\Delta_i$ is $H_{m_i}$-distinguished for $2k+1\leq i\leq r$.
By Theorem \ref{thm:classificationH-distinguished}, $\pi$ is $H$-distinguished, \textit{i.e.} (B) holds.
\end{proof}

\begin{rem}
The converse direction (B)$\Rightarrow$(A) does not hold in general.
In fact, suppose $m$ is odd and $\Delta$ is a $\GL_m(F)$-distinguished irreducible essentially square-integrable representation of $G_m$.
Then, a representation $\pi=\Delta\times\Delta$ of $G_{2m}$ is $H_m$-distinguished and does not satisfy the condition (A).
\end{rem}

\textbf{Acknowledgement.}
I would like to thank Tamotsu Ikeda, Hiraku Atobe and Masao Oi for giving me many helpful comments.
I am also grateful for the anonymous referees for pointing out some inaccuracies in an earlier draft, especially the equation (\ref{Jac}) in Lemma \ref{Jacquet} and the numerous comments which greatly improve the manuscript.
This research does not receive any specific grant from funding agencies in the public, commercial, or not-for-profit sectors.

\end{document}